\documentclass{amsart}
\usepackage{stmaryrd}
\usepackage{silence}
\usepackage{upgreek}
\usepackage{amssymb}
\usepackage{faktor}
\usepackage[all,cmtip]{xy}
\usepackage{amsmath} 
\usepackage{mathrsfs}
\usepackage{tikz}
\usepackage{tikz-cd}
\usepackage{enumitem}
\usepackage{mathtools}
\usepackage[mathcal]{euscript}
\usepackage{graphicx}
\usepackage{url}
\usepackage{hyperref}
\usepackage[T1]{fontenc}
\usepackage{bbm}

\usetikzlibrary{shapes.geometric}
\usetikzlibrary{decorations.markings}
\usetikzlibrary{patterns,decorations.pathreplacing}

\usepackage{ragged2e}

\definecolor{coloryellow}{RGB}{240,228,66}
\definecolor{colorskyblue}{RGB}{86,180,233}
\definecolor{colorvermillion}{RGB}{213,94,0}

\DeclareMathOperator{\Hom}{Hom}

\DeclareSymbolFont{sfletters}{OT1}{cmss}{m}{n}
\DeclareMathSymbol{\sTheta}{\mathord}{sfletters}{"02}

\theoremstyle{definition}
\newtheorem{definition}{Definition}[section]

\newtheorem{example}[definition]{Example}

\theoremstyle{plain}
\newtheorem{proposition}[definition]{Proposition}
\newtheorem{lemma}[definition]{Lemma}
\newtheorem{corollary}[definition]{Corollary}
\newtheorem{theorem}[definition]{Theorem}

\theoremstyle{remark}
\newtheorem{remark}[definition]{Remark}

\makeatletter
\@addtoreset{definition}{section}
\makeatother

\usepackage{marginnote}
    \DeclareFontFamily{U}{wncy}{}
    \DeclareFontShape{U}{wncy}{m}{n}{<->wncyr10}{}
    \DeclareSymbolFont{mcy}{U}{wncy}{m}{n}
    \DeclareMathSymbol{\Sha}{\mathord}{mcy}{"58}

\newsavebox{\foobox}

\title{Configuration spaces and the Arone--Mahowald theorem}
\author{Ben Knudsen}
\author{Dezhou Li}
\begin{document}
\begin{abstract}
We take up the study, initiated by Fred Cohen, of the Cartan--Leray spectral sequence for Euclidean configuration spaces, establishing a combinatorial decomposition as a direct sum of simpler spectral sequences. As an immediate consequence, we recover a difficult theorem of Arone--Mahowald on the vanishing of Goodwillie derivatives of the identity functor.
\end{abstract}
\maketitle

\section{Introduction}

In Cohen's landmark study of the mod $p$ cohomology of Euclidean configuration spaces \cite[III]{CohenLadaMay:HILS}, the starting point is a spectral sequence. Specifically, writing $F_k=F_k(\mathbb{R}^n)$ for the configuration space of $k$ ordered points in $n$-dimensional Euclidean space (resp. $B_k$, unordered), Cohen considers the Cartan--Leray spectral sequence $E(k)$ for the cover $F_k\to B_k$. In a beautiful argument, Cohen completely determines the structure of $E(p)$, but, rather than continuing, he uses this narrow result as a seed calculation in an indirect approach, which leverages detailed knowledge of the homology of iterated loop spaces. The direct approach via the spectral sequence he pronounces possible in principle but ``formidably complicated.''

The purpose of this work and its intended sequel is to take up this challenge. Our program is motivated by the fact that Cohen's calculation, organized according to operadic principles, obscures a different algebraic structure in cohomology, arising from coordinate projections among ordered configuration spaces. The Cartan--Leray spectral sequence is well suited to the study of this structure, which we aim to systematize and exploit in future work.

In order to state the results of the present paper, fix $n>1$ and a prime $p$. Already implicit in Cohen's work is a decomposition of the second page of the spectral sequence in terms of certain subrepresentations $M_\lambda\subseteq H^*(F_k;\mathbb{F}_p)$, taking the form
\[E(k)_2\cong \bigoplus_{\lambda\vdash k}H^*(\Sigma_k; M_\lambda),\] which we show is reflective of a decomposition of spectral sequences. We will say that the partition $\lambda\vdash k$ is \emph{subordinate} to $\mu\vdash k$, written $\lambda\preccurlyeq\mu$, if $\lambda$ is obtained from $\mu$ by splitting a block of size $pr$ into $p$ blocks of size $r$, perhaps multiple times. 

\begin{theorem}[Decomposition theorem]\label{thm:decomposition}
If $\mu\vdash k$ is maximal under subordinacy, then there is a spectral subsequence $E(k,\mu)\subseteq E(k)$ with second page $E(k,\mu)_2\cong \bigoplus_{\lambda\preccurlyeq\mu} H^*(\Sigma_k; M_\lambda)$. Moreover, there is an isomorphism of spectral sequences
\[E(k)\cong\bigoplus_{\mu}E(k,\mu),\] where $\mu$ ranges over partitions of $k$ maximal under subordinacy.
\end{theorem}

The representations $M_\lambda$ have deep algebraic, combinatorial, and homotopical significance, some of which are explored in \cite[\S2]{HershReiner:RSCCSR}, where they are denoted $\mathrm{Lie}_\lambda$ for $n$ odd and $\mathrm{W}_\lambda$ for $n$ even. In algebra, they are intimately tied to Lie theory, being induced from tensor products of various Lie representations; in combinatorics, they are closely connected to the Whitney homology of partition lattices; and, in homotopy theory, they appear in the study of the Goodwillie derivatives of the identity functor, evaluated on spheres. In this latter connection, they are the subject of a difficult calculation of Arone--Mahowald \cite{AroneMahowald:GTIFUPHS}, which essentially falls out of our decomposition theorem. 

As a matter of terminology, we say that $\lambda\vdash k$ is \emph{special} for $n$ odd if the size of every block is a power of $p$; for $n$ even, we also allow blocks of size twice a power of $p$.

\begin{corollary}[Arone--Mahowald theorem]\label{cor:am}
If $\lambda\vdash k$ is not special, then, for every $s>0$, we have \[H^s(\Sigma_k;M_\lambda)=0.\]
\end{corollary}

In future work, we will further exploit the structure of the Cartan--Leray spectral sequence to determine $H^*(\Sigma_k;M_\lambda)$ in the non-vanishing cases, recovering the descriptions of mod $p$ spectral Lie algebra power operations given in \cite{AntolinCamarena:MTHFSLA, Kjaer:OPHFAOSLO}.

\subsection{Conventions} We work in cohomological grading, and graded objects are concentrated in non-negative degrees. We write $\langle k \rangle=\{1,\ldots k\}$, and we write $\Sigma_k$ for the group of bijections from this set to itself. We assume basic familiarity with the theory of group cohomology, a general reference for which is \cite{Brown:CG}.

\subsection{Acknowledgments}
The first author thanks Haynes Miller for his support in pursuing this long-and-still-gestating project; Dev Sinha for conversation, inspiration, and for writing \cite{Sinha:HLDO}, without which this paper would not exist; and Mike Hopkins, Jacob Lurie, and everyone in the 2017 Thursday seminar at Harvard, where he first learned of the Arone--Mahowald theorem. He was supported during the writing of this paper by NSF grant DMS-2551600.

\section{Spectral sequences}

Our purpose here is to develop the basic tools concerning the construction and manipulation of spectral sequences on which our later arguments will rely. We begin with a general construction in group cohomology specializing to a number of examples critical to our main argument. Although this material is well known---see \cite[VII]{Brown:CG}, for example---we have chosen a relatively self-contained exposition. We next turn to the question of understanding how and when we may split summands off of spectral sequences, a question of obvious relevance to Theorem \ref{thm:decomposition}. Throughout, we assume basic familiarity with spectral sequences and direct the reader to \cite{McCleary:UGSS} for reminders.

\subsection{Spectral sequences in group cohomology}

Throughout this section, we fix a group $G$, not necessarily finite, and an arbitrary commutative ground ring $R$, not necessarily a field, which we mostly suppress from the notation. Given a cochain complex $C$ of $G$-modules, we write $H^*(G;C)$ for the cohomology of the bicomplex $\Hom_G(W, C)$, where $W$ is any free resolution of the ground ring as a (trivial) $G$-module. 

All of the spectral sequences considered in this paper are derived from the following basic construction.

\begin{proposition}\label{prop:general ss}
Given a cochain complex $C$ of $G$-modules, there are spectral sequences of the form
\begin{align*}
E_2^{s,t}&\cong H^s(G;H^t(C))\implies {H}^{s+t}(G;C)\\
E_1^{s,t}&\cong H^t(G; C^s)\implies {H}^{s+t}(G;C).
\end{align*}
\end{proposition}
\begin{proof}
For the first spectral sequence, we consider the spectral sequence for the bicomplex obtained by first taking cohomology with respect to the differential in $C$. Since $W$ is free, the functor $\Hom_G(W^s,-)$ is exact for each $s\geq0$, so we have $E_1^{*,t}\cong \Hom_G(W,H^t(C))$ as cochain complexes, implying the claimed description of the second page. The second spectral sequence is obtained by first taking cohomology with respect to the differential in $W$.
\end{proof}

From this general result, we derive our primary tool, the Cartan--Leray spectral sequence.

\begin{corollary}[Cartan--Leray]\label{cor:cartan leray}
Let $q:E\to B$ be a regular covering space with deck group $G$. There is a spectral sequence of the form
\[E_2^{s,t}\cong H^s(G; H^t(E))\implies H^{s+t}(B).\]
\end{corollary}
\begin{proof}
We apply Proposition \ref{prop:general ss} to the complex $C=C^*(E)$. The $E_2$-page of the first spectral sequence is as desired; to identify the $E_\infty$-page, we consult the second spectral sequence. By our assumption on $q$, the action of $G$ is free on $E$, and hence on $C_s(E)$ for each $s\geq0$, so $C^s(E)$ is an injective $G$-module. Thus, the first page of this second spectral sequence is concentrated on the $t=0$ row, which takes the form
\[E_1^{*,0}\cong C^*(E)^G\simeq C^*(B),
\] where we have used that $G$ acts freely on $E$. Thus, the second spectral sequence collapses to $H^*(B)$ which is necessarily also the target of the first.
\end{proof}

We will also have use for an analogue of the Cartan--Leray spectral sequence for compactly supported cohomology. 

\begin{corollary}\label{cor:cs cartan leray}
Let $G$ be a finite group acting freely on the orientable manifold $M$, and write $N=M/G$. There is a spectral sequence of the form
\[E_2^{s,t}\cong H^s(G; H_c^t(M; d^*\mathrm{sgn}))\implies H_c^{s+t}(N; \mathbb{Z}^w),\] where $d:G\to C_2$ is the degree homomorphism induced by the action of $G$ on $M$, $\mathrm{sgn}$ denotes the sign representation, and $\mathbb{Z}^w$ denotes the orientation local system.
\end{corollary}

For the proof, we will require the following simple topological observation.

\begin{lemma}\label{lem:orientation cover}
In the situation of Corollary \ref{cor:cs cartan leray}, there is a canonical isomorphism of covering spaces of the form
\[\widetilde N\cong M\times_G C_2,\] where $\widetilde N$ denotes the orientation double cover of $N$.
\end{lemma}
\begin{proof}
Since $M$ is oriented, we have $\widetilde M=M\times C_2$ as $G$-spaces over $M$, and the canonical map $\widetilde M\to \widetilde N$ induces a homeomorphism $\widetilde M/G\cong \widetilde N$.
\end{proof}

\begin{proof}[Proof of Corollary \ref{cor:cs cartan leray}]
We follow the strategy of Corollary \ref{cor:cartan leray}, applying Proposition \ref{prop:general ss} instead to the complex $C=\varinjlim_K C^*(M,M\setminus \bar K;d^*\mathrm{sgn})$, where $K$ ranges over compact subsets of $N$, and we write $\bar K$ generically for the preimage of $K$ under whichever projection map is relevant. Since $G$ is finite, the projection to the quotient is proper, so every compact subset of $M$ is contained in one of the form $\bar K$, and the identification of the $E_2$-page of the first spectral sequence follows. On the other hand, since $H^*(G;-)$ commutes with direct limits for $G$ finite, the second spectral sequence has $E_1$-page given by 
{\small\begin{align*}
\varinjlim_KH^*(G; C^*(M,M\setminus \bar K;d^*\mathrm{sgn}))
&\cong \varinjlim_K\Hom_{G}(C_*(M,M\setminus \bar K),d^*\mathrm{sgn})\\
&\simeq \varinjlim_K\Hom_{G\times C_2}(C_*(M\times C_2,(M\setminus \bar K)\times C_2),d^*\mathrm{sgn}\otimes\mathrm{sgn})\\
&\simeq\varinjlim_K\Hom_{ C_2}(C_*(M\times_G C_2, (M\setminus \bar K)\times_G C_2),\mathrm{sgn})\\
&\cong \varinjlim_K\Hom_{ C_2}(C_*(\widetilde N, \widetilde N\setminus \bar K),\mathrm{sgn})
\end{align*}}where the first line follows from the freeness assumption as before, the second and third from basic universal properties and the fact that $C_*(X/G)\simeq C_*(X)_G$ for free $G$-spaces $X$, and the fourth from Lemma \ref{lem:orientation cover}. It follows that the second spectral collapses to $H_c^*(N;\mathbb{Z}^w)$, and the claim follows as before.
\end{proof}

The same method establishes the following purely algebraic result.

\begin{corollary}[Lyndon--Hochschild--Serre]\label{cor:lhs}
Given a normal subgroup $K\unlhd G$ and a $G$-module $M$, there is a spectral sequence of the form
\[
E_2^{s,t}\cong H^s(G/K; H^t(K;M))\implies H^{s+t}(G;M).
\]
\end{corollary}
\begin{proof}
We apply Proposition \ref{prop:general ss} to the complex $C=\Hom_K(W,M)$ of $G/K$-modules, where $W$ is as above. Since $W$ is equally a free resolution of the ground ring as an $K$-module, the $E_2$-page of the first spectral sequence is as desired. On the other hand, the isomorphism
\[\Hom_K(R[G],M)\cong \prod_{[g]\in G/K}\Hom_K(g\cdot R[K],M)\] implies that this $G/K$-module is coinduced, hence injective, so the first page of the second spectral sequence is concentrated on the $t=0$ row, which takes the form
\[E_1^{*,0}\cong \Hom_K(W,M)^{G/K}\cong \Hom_G(W,M).\] Thus, the second spectral sequence collapses to $H^*(G;M)$, implying the claim as before.
\end{proof}

It is clear that the construction of Proposition \ref{prop:general ss} is natural with respect to $G$-equivariant chain maps. We close this section with an observation concerning the naturality of this type of spectral sequence with respect to the group.

\begin{proposition}\label{prop:naturality}
Given a subgroup $K\leq G$ of finite index and a complex $C$ of $G$-modules, there are maps of spectral sequences of the form
\[\xymatrix{
H^s(K;H^t(C))\ar@{=>}[d]\ar[r]^-{\mathrm{cor}_K^G}&H^s(G;H^t(C))\ar@{=>}[d]\ar[r]^-{\mathrm{res}_K^G}&H^s(K;H^t(C))\ar@{=>}[d]\\
H^{s+t}(K;C)\ar[r]^-{\mathrm{cor}_K^G}&H^{s+t}(G;C)\ar[r]^-{\mathrm{res}_K^G}&H^{s+t}(K;C)
}\]
\end{proposition}

The key observation in the proof is that our spectral sequences are compatible with Shapiro's lemma.

\begin{lemma}\label{lem:shapiro ss}
Given a subgroup $K\leq G$ of finite index and a complex $C$ of $K$-modules, there is an isomorphism of spectral sequences of the form
\[
\xymatrix{
H^s(K;H^t(C))\ar@{=>}[d]\ar[r]^-{\cong}&H^s(G;\mathrm{Coind}_K^G\,H^t(C))\ar@{=>}[d]\\
H^{s+t}(K;C)\ar[r]^-{\cong}&H^{s+t}(G;\mathrm{Coind}_K^G\,C),
}
\] where the indicated isomorphisms on $E_2$ and $E_\infty$ are those of Shapiro's lemma.
\end{lemma}
\begin{proof}
The two spectral sequences arise by applying Proposition \ref{prop:general ss} to $(K, C)$ and $(G, \mathrm{Coind}_K^G\,C)$, respectively. Since the functor $H^t$ commutes with coinduction by our assumption on $[G:K]$, the righthand spectral sequence has the claimed $E_2$-page. The isomorphism between the two is induced by the isomorphism $\Hom_G(W, \mathrm{Coind}_K^G\, C)\cong \Hom_K(W,C)$, where we use that $W$ is also free over $K$.
\end{proof}

\begin{proof}[Proof of Proposition \ref{prop:naturality}]
The righthand map of spectral sequences is obtained by composing the isomorphism of Lemma \ref{lem:shapiro ss} with the map induced by the unit map $C\to \mathrm{Coind}_K^G\mathrm{Res}_K^G\,C$. The lefthand map is obtained in a similar manner, instead using the counit $\mathrm{Ind}_K^G\mathrm{Res}_K^G\,C\to C$ after passing through the isomorphism between induction and coinduction.
\end{proof}

\subsection{Spectral subsequences} As our main result concerns identifying summands in a spectral sequence, we will require criteria for identifying such summands. The purpose of this section is to develop such criteria. We begin by articulating a reasonable notion of a subobject in this context.

\begin{definition}
Let $E$ be a spectral sequence. A \emph{spectral subsequence} of $E$ is a collection $E'=\{E'_r\}$ such that
\begin{enumerate}
\item $E_r'$ is a subcomplex of $E_r$, and
\item the induced map $H^*(E_r')\to H^*(E_r)\cong E_{r+1}$ is an isomorphism onto $E_{r+1}'$.
\end{enumerate}
\end{definition}

In this situation, $E'$ forms a spectral sequence, and the inclusions form a map of spectral sequences $E'\to E$.

\begin{definition}
We say that a submodule of $E_2$ \emph{spans a spectral subsequence} if it is the second page of some spectral subsequence of $E$.
\end{definition}

It is easy to see that a submodule of $E_2$ spans at most one spectral subsequence, but of course it may not span any. In order to characterize which submodules do span spectral subsequences, we require the following concept.

\begin{definition}
Let $C$ be a cochain complex with differential $d$. We call a subcomplex $A$ \emph{independent} if $dx \in A$ if and only if $dx = dy$ for some $y \in A$.
\end{definition}

An immediate result is that the subcomplex $A$ is independent if and only if the induced homomorphism $H^*(A) \rightarrow H^*(C)$ is injective.

\begin{definition}\label{spectralsub}
Let $\{E_r\}$ be a spectral sequence and $E'_2$ a submodule of $E_2$. We define what it means for $E'_2$ to be \emph{independent through $E_r$} recursively as follows:
\begin{enumerate}
\item First, we say that $E_2'$ is \emph{independent through $E_2$} if $E'_2$ is an independent subcomplex of $E_2$. In this case, we define $E'_3 := H^*(E'_2)$, thought of as a submodule of $E_3$.
\item Assume that the term \emph{independent through $E_{r-1}$} and the subspaces $E_r'\subseteq E_r$ have been defined. We say that $E'_2$ is \emph{independent through $E_r$} if $E_r'$ is an independent subcomplex of $E_r$.
\end{enumerate}
\end{definition}

We now relate this inductive notion of independence to our earlier concept of a spectral subsequence.

\begin{lemma}\label{lem:spanning criterion}
A submodule of $E_2$ spans a spectral subsequence if and only if it is independent through $E_r$ for every $r\geq2$.
\end{lemma}
\begin{proof}
If $E_2'\subseteq E_2$ is independent through $E_r$ for every $r\geq2$, then $E'=\{E_r'\}$ is a spectral subsequence of $E$. Conversely, if $E_2'$ spans a spectral subsequence $E'$, then $E_2'$ is in particular a subcomplex and the map $H^*(E_2')\to H^*(E_2)\cong E_3$ is an isomorphism onto $E_3'$; in particular, this map is injective, so $E_2'$ is independent. Applying the same reasoning inductively yields the desired conclusion.
\end{proof}

We come now to the main result of the section, which we will use below in identifying summands in the Cartan--Leray spectral sequence.

\begin{proposition}\label{prop:projection extends}
Let $E$ be a spectral sequence and $E_2\cong V\oplus W$ a direct sum decomposition. The following conditions are equivalent.
\begin{enumerate}
\item The projection onto $V$ extends to a map $\pi:E\to E$ of spectral sequences.
\item There is a map $\rho:E\to E$ of spectral sequences such that $\rho|_{V}$ is an isomorphism onto $V$ and $\rho|_W=0$.
\item The projection onto $W$ extends to a map $\pi:E\to E$ of spectral sequences.
\item There is a map $\rho:E\to E$ of spectral sequences such that $\rho|_{W}$ is an isomorphism onto $W$ and $\rho|_V=0$.
\item The subspaces $V$ and $W$ span spectral subsequences $E'$ and $E''$.
\end{enumerate}
Moreover, in this situation, we have $E\cong E'\oplus E''$.
\end{proposition}
\begin{proof}
The first condition implies the third, since $1-\pi$ is a map of spectral sequences extending the projection onto $W$, and the third implies the first by symmetry. On the other hand, assuming the fifth condition, the resulting map of spectral sequences $E'\oplus E''\to E$ is an isomorphism on $E_2$ by assumption, hence an isomorphism by Zeeman's theorem, and the composite $E\cong E'\oplus E''\to E'\subseteq E$ is a map of spectral sequences extending the projection onto $V$ (resp. $E''$, $W$). Thus, the fifth condition implies the first. Since the first obviously implies the second and the third the fourth, it suffices by symmetry to show that the second implies the fifth.

Assuming the second condition, we have $d_2(\rho(v))=\rho(d_2(v))\in V$ for any $v\in V$, so $V$ is a subcomplex. On the other hand, for $w\in W$, we have $\rho(d_2(w))=d_2(\rho(w))=0$, so $d_2(w)\in W$, and $W$ is also a subcomplex. It follows that both $V$ and $W$ are independent, and $E_2\cong V\oplus W$ as complexes, so $E_3\cong H^*(V)\oplus H^*(W)$. Applying the same reasoning inductively shows that both $V$ and $W$ are independent through $E_r$ for every $r\geq0$, hence span spectral subsequences by Lemma \ref{lem:spanning criterion}.
\end{proof}

\section{Cohomology of configuration spaces}

We turn to our main subject of interest, namely the cohomology of the spaces $F_k=F_k(\mathbb{R}^n)$ and $B_k=B_k(\mathbb{R}^n)$ for fixed $n>1$, building on the foundational work of Arnold \cite{Arnold:CRCBG} and Cohen \cite[III]{CohenLadaMay:HILS}, which Section \ref{section:arnold-cohen} below reviews in a form suitable for our purposes. For varying $k$, these spaces are interrelated by a host of maps, which we may divide into two categories: the 
\emph{maps in}, induced by embeddings, and the \emph{maps out}, induced by the coordinate projections. These categories of map have each received extensive study, the key words being \emph{operad} in the former case \cite{May:GILS} and \emph{twisted commutative algebra} or \emph{FI-module} in the latter \cite{ChurchEllenbergFarb:FIMSRSG}. As it will be necessary for us to mix the two types of map, we devote Section \ref{section:maps in maps out} to the study of their interactions, the main result being Proposition \ref{prop:endomorphism constant}. 

\subsection{Compositions and partitions}

We begin by detailing some necessary conventions around the combinatorics governing the maps of interest to us.

\begin{definition}
Let $I$ be a finite set. A \emph{set composition} of $I$ of \emph{length $\ell$} is a surjection $\pi:I \to \langle \ell\rangle$. The \emph{blocks} of the set composition $\pi$ are the subsets $\pi_i:=\pi^{-1}(i)\subseteq I$ for $1\leq i\leq \ell$.
\end{definition}

Equivalently, we may view a set composition as a tuple of pairwise disjoint, non-empty subsets whose union is $I$, or as an equivalence relation on $I$. We will pass freely between these perspectives.

In the case $I=\langle k\rangle$ of greatest interest, we speak of a set composition of $k$. This concept gives rise to the three further concepts of a \emph{composition} of $k$, a \emph{set partition} of $k$, and a \emph{partition} of $k$. The two rules for relating these concepts are that dropping the word ``set'' indicates passing to orbits for the right action of $\Sigma_k$ on the set of surjections, while substituting ``partition'' for ``composition'' indicates passing to orbits for the left action of $\Sigma_\ell$. 

Each of these concepts has an attendant concept of length and block, which are inherited in a straightforward manner from the originals, the difference being that the blocks of a (set) partition are unordered, and the blocks of a composition or partition are positive numbers summing to $k$ rather than subsets of $\langle k\rangle$. We use the conventional notation $\lambda\vdash k$ to indicate that $\lambda$ is a partition of $k$. 

Writing $\widetilde \Pi_{\langle k\rangle}$ for the set of all set compositions of $k$, the relationships among these concepts are summarized in the commuting diagram
\[\xymatrix{
\widetilde \Pi_{\langle k\rangle}\ar[r]\ar[d]&\Pi_{\langle k\rangle}\ar[d]\\
\widetilde \Pi_{ k}\ar[r]& \Pi_{ k}
}\] of canonical projections. In this notation, dropping the brackets indicates dropping the word ``set,'' and dropping the tilde indicates substituting ``partition'' for ``composition.'' The bottom and lefthand arrows in this diagram each admit a canonical section, the first by ordering the blocks of a partition in non-increasing order, the second by requiring the ordering on blocks to respect the natural ordering on $\langle k\rangle$. By composition, it follows that the righthand arrow also admits a canonical section. In what follows, our notation often fails to distinguish an element of one of these sets from its image under one of the various functions among them. Thus, for example, we will often use the same symbol $\lambda$ in different contexts to denote a partition and the corresponding set composition. 

Under the action of $\Sigma_k$ on the top row of the above diagram, we write $P_\pi$ for the stabilizer of the set composition $\pi$ and $W_\alpha$ for the stabilizer of the set partition $\alpha$. Following the abusive convention above, given a partition $\lambda\vdash k$, we write $P_\lambda$ and $W_\lambda$ for the stabilizers of the corresponding set composition and set partition, respectively. Our notation is mnemonic, as $P_\lambda$ is a product and $W_\lambda$ a wreath product of symmetric groups.\footnote{The subgroup $P_\lambda$ is otherwise known as the Young subgroup associated to $\lambda$.} Note that, since $W_\lambda$ permutes the blocks of the set composition corresponding to $\lambda$, there results a canonical action of $W_\lambda$ on $\langle \ell(\lambda)\rangle$.

\begin{example}
The tuples $\pi_1=(\{1,2\}, \{3,4\})$ and $\pi_2=(\{1,3\}, \{2,4\})$ are distinct set compositions of $4$, which have distinct underlying set partitions $\rho_1=\left\{\left\{1,2\right\}, \left\{3,4\right\}\right\}$ and $\rho_2=\left\{\left\{1,3\right\}, \left\{2,4\right\}\right\}$, respectively. Both have underlying composition $\eta=(2,2)$ and partition $\lambda=\{2,2\}$ (a multiset). As for the canonical sections, the composition associated to $\lambda$ is $\eta$, and the set composition associated to $\eta$ is $\pi_1$. The subgroup $P_{\pi_1}$ is the image in $\Sigma_4$ of the standard inclusion of $\Sigma_2\times\Sigma_2$, while $P_{\pi_2}$ is the conjugate of this subgroup by the transposition interchanging $2$ and $3$ (resp. $W_{\pi_1}$, $\Sigma_2\wr \Sigma_2$, $W_{\pi_2}$).
\end{example}

\subsection{The Arnold--Cohen theorem}\label{section:arnold-cohen} In this section and the next, cohomology is taken with integer coefficients. As indicated above, we fix $n>1$ and make the abbreviation $F_k=F_k(\mathbb{R}^n)$.

To begin, we define a cohomology class $\alpha_{ab}\in H^{n-1}(F_k)$ by pulling back the canonical generator of $H^{n-1}(S^{n-1})$ via the Gauss map 
\[\gamma_{ab}(x_1,\ldots, x_k)=\frac{x_b-x_a}{|x_b-x_a|},\] where $1\leq a\neq b\leq k$. As the following result attests, these classes are the fundamental building block in the cohomology of Euclidean configuration spaces. Although this calculation is classical, it will be convenient to indicate some aspects of its proof for later use, but we will be brief. The original references are \cite{Arnold:CRCBG} in the case $n=2$ and \cite[III.7]{CohenLadaMay:HILS} in general, and expository accounts can be found in \cite{Sinha:HLDO, Knudsen:CSAT}, for example.

\begin{theorem}[Arnold--Cohen]\label{thm:arnold-cohen}
The graded commutative ring $H^*(F_k)$ is generated by the classes $\alpha_{ab}$ subject only to the following relations:
\begin{enumerate}
\item $\alpha_{ab}^2=0$
\item $\alpha_{ab}=(-1)^n\alpha_{ba}$
\item $\alpha_{ab}\alpha_{bc}+\alpha_{bc}\alpha_{ca}+\alpha_{ca}\alpha_{ab}=0$.
\end{enumerate}
\end{theorem}
\begin{proof}
According to \cite{FadellNeuwirth:CS}, the map $F_k\to F_{k-1}$ is a fiber bundle with fiber over $(x_1,\ldots, x_{k-1})$ homeomorphic to $\mathbb{R}^n\setminus\{x_1,\ldots, x_{k-1}\}$. Fixing $k$ and writing $\kappa_{j}$ for the fundamental class of the submanifold of $F_k$ in which the $j$th particle lies at fixed distance from the $k$th, and all other particles are stationary at some farther distance, we have $\alpha_{ik}(\kappa_j)=\delta_{ij}$. Thus, the Leray--Hirsch theorem applies to our fiber bundle, and it follows inductively that the cohomology of $F_k$ has Poincar\'{e} polynomial $\prod_{\ell=1}^{k-1}(1+\ell t^{n-1})$ and is generated as a ring by the $\alpha_{ab}$. The square-zero and antisymmetry relations shown above are pulled back from the sphere, and the third---the ``Arnold relation''---follows easily from symmetry and the fact that $H^{2(n-1)}(F_3)$ has rank $2$, so we obtain a surjection $\mathcal{A}\to H^*(F_k)$, where $\mathcal{A}$ is the graded commutative algebra of the theorem statement.

Through repeated use of the defining relations of $\mathcal{A}$, one can show that both $\mathcal{A}$ and $H^*(F_k)$ are spanned by monomials of the form 
\[\prod_{r=1}^s \alpha_{a_{1,r}a_{2,r}}\alpha_{a_{2,r}a_{3,r}}\cdots\alpha_{a_{m(r)-1,r}a_{m(r),r}}\] such that $a_{1,r}<a_{j,r}$ for each $j$, the sum of the $m(r)$ is at most $k$, the sets $\{a_{j,r}\}_{j=1}^{m(r)}$ are pairwise disjoint, and $a_{1,r}<a_{1,r+1}$ for each $r$---see \cite[Lem. 4.4]{Sinha:HLDO}. In fact, this collection forms a basis for $H^*(F_k)$, since an easy recursion shows that the free graded Abelian group it generates has the same Poincar\'{e} polynomial as above. Thus, under the homomorphism $\mathcal{A}\to H^*(F_k)$, a spanning set for the source is mapped to a basis for the target, implying the claim.
\end{proof}

Toward an understanding of $H^*(F_k)$ as a $\Sigma_k$-module, we observe that a square-free monomial $\alpha_{a_1b_1}\cdots\alpha_{a_mb_m}$ determines an equivalence relation on $k$---setting  $a_i\sim b_i$ and imposing transitivity---and hence a set partition. The relations of graded commutativity and antisymmetry preserve this set partition, as does the Arnold relation. While the action of $\Sigma_k$ does not, it does preserve the underlying partition of $k$. In this way, we obtain a $\Sigma_k$-submodule $M_\lambda\subseteq H^*(F_k)$ for each partition $\lambda$ of $k$ (note that this notation, as elsewhere, does not reflect the dependence on $n$). Notice that $M_\lambda$ lies in degree $d(\lambda):=(n-1)(k-\ell(\lambda))$. In the simplest case $\lambda=(k)$, we write $M_k:=M_{(k)}=H^{(n-1)(k-1)}(F_k)$.

Examining the basis element shown above, we observe that it lies in $M_\lambda$ if and only if the multiset $\{m(r)\}_{r=1}^s$ coincides with the multiset of blocks of $\lambda$ of size greater than $1$. In particular, each such basis element lies in some $M_\lambda$.

\begin{corollary}\label{cor:partition decomposition}
For any $n>1$ and $k\geq0$, we have the direct sum decomposition of $\Sigma_k$-modules
\[H^*(F_k)\cong \bigoplus_{\lambda \vdash k} M_\lambda.\]
\end{corollary}

\begin{corollary}\label{cor:cyclic and dimension}
The module $M_\lambda$ is free of rank $\frac{k!}{|W_\lambda|}\prod_{i=1}^{\ell(\lambda)}(\lambda_i-1)!$ as an Abelian group and cyclic as a $\Sigma_k$-module.
\end{corollary}
\begin{proof}
Consider a basis element lying in $M_\lambda$. The choice of the disjoint subsets $\{a_{j,r}\}_{j=1}^{m(r)}$ is equivalent to the choice of a set partition with underlying partition $\lambda$---the remaining blocks are singletons---of which there are $\frac{k!}{|W_\lambda|}$ by the orbit-stabilizer theorem. The remaining ambiguity is the choice for each $r$ of an ordering of the $a_{j,r}$ with $j>1$. Since the multiset of the $m(r)$ coincides with the multiset of blocks of $\lambda$ of size greater than $1$, the first claim follows. For the second claim, we observe that any two basis elements lying in $M_\lambda$ differ by a permutation.
\end{proof}

\subsection{Maps in and maps out}\label{section:maps in maps out}

In this section, we study two classes of maps between (products of) configuration spaces. To begin, given a set composition $\pi=(\pi_1,\ldots, \pi_\ell)$ of $k$, we write $F_\pi=\prod_{i=1}^{\ell} F_{|\pi_i|}$. This space carries a canonical action of $W_\pi$.

The simpler of the two types of map is the map $p_\pi: F_k\to F_\pi$ induced by the coordinate projections; we refer to this type of map as a ``map out''. There is also a ``map in'' of the form $i_\pi:F_\pi\to F_k$, which is given by the composite
\[F_\pi\to F_k\left(\sqcup_{\ell}\mathbb{R}^n\right)\to F_k(\mathbb{R}^n),\] where the first arrow is the inclusion of the source as the subspace of the target in which particles labeled by elements of the $i$th block of $\pi$ lie in the $i$th component of the disjoint union, and the second arrow is induced by any orientation preserving embedding. Since the space of such embeddings is path connected, this choice affects $i_\pi$ only up to homotopy. An important point is that, while the map out $p_\pi$ is $W_\pi$-equivariant, the map in $i_\pi$ is only $P_\pi$-equivariant.

In addition to maps in and maps out, there are symmetries. Specifically, given a permutation $\sigma\in \Sigma_k$, there is a canonical homeomorphism $\sigma_\pi:F_{\pi\sigma^{-1}}\to F_\pi$ fitting into the commutative diagram
\[\xymatrix{
F_k\ar[d]_-{p_{\pi\sigma^{-1}}}\ar[r]^-\sigma&F_k\ar[d]^-{p_\pi}\\
F_{\pi\sigma^{-1}}\ar[r]^-{\sigma_\pi}&F_\pi.
}\] 

Our first main result shows that the modules $M_\lambda$, which are the building blocks of the cohomology of $F_k$ by Corollary \ref{cor:partition decomposition}, are determined representation theoretically by our maps out.

\begin{proposition}\label{prop:induction description}
For any $\lambda\vdash k$, the homomorphism $p_\lambda^*: H^{*}(F_\lambda)\to H^*(F_k)$ induces an isomorphism 
\[M_\lambda\cong \mathrm{Ind}_{W_\lambda}^{\Sigma_k} H^{d(\lambda)}(F_\lambda).\]
\end{proposition}
\begin{proof}
Since the image of $H^{d(\lambda)}(F_\lambda)$ under $p_\lambda^*$ clearly lies in $M_\lambda$,  and since $p_\lambda$ is $W_\lambda$-equivariant, we obtain a map from right to left from the universal property of induction. As shown in Corollary \ref{cor:cyclic and dimension}, the $\Sigma_k$-module $M_\lambda$ is generated by any of its basis elements, and it is easy to see that the image of $p_\lambda^*$ contains a basis element. Since source and target are free Abelian, it suffices by the same result to observe that the righthand side has rank $\frac{k!}{|W_\lambda|}\prod_{i=1}^{\ell(\lambda)}(\lambda_i-1)!$, since $H^{d(\lambda)}(F_\lambda)\cong\bigotimes_{i=1}^{\ell(\lambda)}H^{\lambda_i-1}(F_{\lambda_i})$ by the K\"{u}nneth theorem, and $H^{\lambda_i-1}(F_{\lambda_i})$ has rank $(\lambda_i-1)!$ by Corollary \ref{cor:cyclic and dimension}.
\end{proof}

Our second  main result concerns the result of composing maps in and maps out at the level of group cohomology. In order to state this result, we introduce the following terminology.

\begin{definition}\label{def:fitting}
Let $\rho$ and $\pi$ be set compositions of $\langle k\rangle$. We say that $\sigma\in \Sigma_k$ is a \emph{fitting} of $\rho$ into $\pi$ if the (necessarily unique) dashed filler $d_\sigma$ exists in the following commuting diagram of sets:
\[\xymatrix{
\langle k\rangle\ar[d]_-\rho \ar[r]^-\sigma&\langle k\rangle\ar[d]^-\pi\\
\langle \ell(\rho)\rangle\ar@{-->}[r]^-{d_\sigma}& \langle \ell(\pi)\rangle
}\] We write $\Sigma_{\rho,\pi}\subseteq\Sigma_k$ for the set of fittings of $\rho$ into $\pi$, and we say that $\rho$ \emph{fits into} $\pi$ if the identity permutation lies in $\Sigma_{\rho,\pi}$. 
\end{definition}

Equivalently, $\rho$ fits into $\pi$ if every block of $\rho$ is contained in some block of $\pi$, and $\sigma$ is a fitting if and only if $\rho\sigma^{-1}$ fits into $\pi$. As a matter of notation, given a set composition $\pi$ of $\langle k\rangle$ and a partition $\lambda\vdash k$, we write
\[c_{\lambda,\pi}=\sum_{P_\pi\sigma W_\lambda,\,\sigma\in \Sigma_{\lambda,\pi}}[W_\lambda:P^{\sigma^{-1}}_\pi\cap W_\lambda].\]

\begin{proposition}\label{prop:endomorphism constant}
Fix a partition $\lambda\vdash k$ and a set composition $\pi$, and consider the composite endomorphism
\[f:H^*(\Sigma_k;H^*(F_k))\to H^*(P_\pi; H^*(F_\pi))\to H^*(\Sigma_k; H^*(F_k)),\] where the first arrow is induced by $i_\pi^*$ and the second by $p_\pi^*$. The restriction of $f$ to $H^*(\Sigma_k;M_\lambda)$ is given by multiplication by $c_{\lambda,\pi}$.
\end{proposition}

For the proof, we require two basic observations.

\begin{lemma}\label{lem:composite identity}\label{lem:identity homotopy}
The composite $p_\pi\circ i_\pi$ is homotopic to the identity.
\end{lemma}
\begin{proof}
By inspection, the map in question breaks into a product over $i\in \langle\ell(\pi)\rangle$ of maps of the form $F_{|\pi_i|}\to F_{|\pi_i|}$, each induced by an orientation preserving self-embedding of $\mathbb{R}^n$. Since the space of such is path connected, each embedding is isotopic to the identity, and the claim follows.
\end{proof}

\begin{lemma}\label{lem:block vanishing}
If $a$ and $b$ lie in different blocks of $\pi$, then $i_\pi^*(\alpha_{ab})=0$.
\end{lemma}
\begin{proof}
Our assumption grants the commutativity up to homotopy of the diagram
\[\xymatrix{
F_\pi\ar[d]\ar[r]^-{i_\pi}&F_k\ar[r]^-{\gamma_{ab}}\ar[d]&S^{n-1}\\
F_1\times F_1\ar[r]^-{i_{\pi'}}&F_2\ar[ur]_-{\gamma_{12}},
}\] where  the vertical arrows are restricted coordinate projections, and $\pi'$ is the set composition with two singleton blocks. Thus, since $F_1$ is contractible, it follows that $\gamma_{ab}\circ i_\pi$ is nullhomotopic, and the claim follows.
\end{proof}

With these observations, we are in a position to prove the following basic result concerning the composition of maps in and maps out, which will a key ingredient in the proof of Proposition \ref{prop:endomorphism constant}.

\begin{lemma}\label{lem:fitting conditions}
Let $\rho$ and $\pi$ be set compositions of $\langle k\rangle$.
\begin{enumerate}
\item If $\rho$ fits into $\pi$, then $p_{\rho}\circ i_\pi\circ p_\pi$ is homotopic to $p_{\rho}$.
\item If $\rho$ does not fit into $\pi$, then the composite $i^*_{\pi}\circ p_{\rho}^*$ vanishes on $H^{d(\rho)}(F_\rho)$.
\end{enumerate}
\end{lemma}
\begin{proof}
For the first claim, our assumption guarantees the existence of the unique dashed filler in the commuting diagram
\[\xymatrix{
F_k\ar[r]^-{p_\pi}\ar[d]_-{p_\rho}&F_\pi\ar@{-->}[dl]^-{p_{\rho,\pi}}\\
F_\rho,
}\] and we have $p_{\rho}\circ i_\pi\circ p_\pi=p_{\rho,\pi}\circ p_\pi\circ i_\pi\circ p_\pi\simeq p_{\rho,\pi}\circ p_\pi=p_\rho$ by Lemma \ref{lem:identity homotopy}. 

For the second claim, we first consider the special case in which $\rho$ has a single block. In this case, the claim is that $i_\pi^*$ vanishes on $H^{(n-1)(k-1)}(F_k)$ for any set composition $\pi$ with more than one block. Appealing to the basis exhibited in Theorem \ref{thm:arnold-cohen}, it suffices to show that $i_\pi^*$ annihilates $\prod_{a=1}^{k-1}\alpha_{\sigma(a)\sigma(a+1)}$ for every $\sigma\in\Sigma_k$ fixing $1$. Since $\pi$ has more than one block, there exists $1< a\leq k$ minimal such that $\sigma(1)$ and $\sigma(a)$ lie in different blocks. Then $\sigma(a-1)$ lies in the same block as $\sigma(1)$ by minimality, hence in a different block from $\sigma(a)$. The claim now follows from Lemma \ref{lem:block vanishing}, which implies that $i_\pi^*(\alpha_{\sigma(a-1)\sigma(a)})=0.$

In the general case, for each $1\leq j\leq \ell(\rho)$, we obtain a set composition $\pi_j$ of $\rho_j$ by intersecting $\rho_j$ with the blocks of $\pi$ and discarding any empty blocks. We then have the following homotopy commutative diagram, in which the vertical arrows are induced by the appropriate coordinate projections:
\[
\xymatrix{
F_\pi\ar[d]\ar[rr]^-{i_\pi}&& F_k\ar[d]\ar[r]^-{p_\rho}& F_\rho\ar@{=}[dl]\\
\displaystyle\prod_{j=1}^{\ell(\rho)} F_{\pi_j}\ar[rr]^-{\prod i_{\pi_j}}&&\displaystyle\prod_{j=1}^{\ell(\rho)} F_{\rho_j}
}
\] Since $\rho$ does not fit into $\pi$, there is some $\pi_j$ with more than one block, and the previous case implies the claim.
\end{proof}

Before embarking on the main argument, we remind the reader that the symbol $\lambda$ may denote both a partition and the associated set composition. 

\begin{proof}[Proof of Proposition \ref{prop:endomorphism constant}]
Explicitly, we have $f=\mathrm{cor}_{P_\pi}^{\Sigma_k}\circ p_\pi^*\circ i^*_\pi\circ \mathrm{res}^{\Sigma_k}_{P_\pi}$. Since $\mathrm{cor}_{W_\lambda}^{\Sigma_k}\circ p_\lambda^*$ is an isomorphism of $H^*(W_\lambda; H^{d(\lambda)}(F_\lambda))$ onto $H^*(\Sigma_k;M_\lambda)$ by Proposition \ref{prop:induction description} and Shapiro's lemma, it suffices to observe the following sequence of equalities between homomorphisms from $H^*(W_\lambda;H^{d(\lambda)}(F_\lambda))$ to  $H^*(\Sigma_k;H^*(F_k))$:
{\begin{align*}
f\circ\mathrm{cor}^{\Sigma_k}_{W_\lambda}\circ p_\lambda^*
&=\sum_{P_\pi\sigma W_\lambda}\mathrm{cor}_{P_\pi}^{\Sigma_k}\circ p_\pi^*\circ i^*_\pi\circ\mathrm{cor}_{P_\pi\cap W_\lambda^\sigma}^{P_\pi}\circ \mathrm{res}^{W_\lambda}_{P_\pi\cap W_\lambda^\sigma}\circ p_\lambda^*\\
&=\sum_{P_\pi\sigma W_\lambda}\mathrm{cor}_{P_\pi}^{\Sigma_k}\circ p_\pi^*\circ i^*_\pi\circ\mathrm{cor}_{P_\pi\cap W_\lambda^\sigma}^{P_\pi}\circ \mathrm{res}^{W_\lambda^\sigma}_{P_\pi\cap W_\lambda^\sigma}\circ p_{\lambda\sigma^{-1}}^*\circ\sigma_\lambda^*\\
&=\sum_{P_\pi\sigma W_\lambda} \mathrm{cor}_{P_\pi\cap W_\lambda^\sigma}^{\Sigma_{k}}\circ p_\pi^*\circ i^*_\pi\circ  \mathrm{res}^{W_\lambda^\sigma}_{P_\pi\cap W_\lambda^\sigma}\circ p_{\lambda\sigma^{-1}}^*\circ\sigma_\lambda^*\\
&=\sum_{P_\pi\sigma W_\lambda} \mathrm{cor}_{P_\pi\cap W_\lambda^\sigma}^{\Sigma_{k}}\circ p_\pi^*\circ i^*_\pi\circ p_{\lambda\sigma^{-1}}^*\circ \mathrm{res}^{W_\lambda^\sigma}_{P_\pi\cap W_\lambda^\sigma}\circ \sigma_\lambda^*\\
&=\sum_{P_\pi\sigma W_\lambda: \,\sigma\in\Sigma_{\lambda,\pi}} \mathrm{cor}_{P_\pi\cap W_\lambda^\sigma}^{\Sigma_{k}}\circ p_{\lambda\sigma^{-1}}^*\circ  \mathrm{res}^{W_\lambda^\sigma}_{P_\pi\cap W_\lambda^\sigma}\circ \sigma_\lambda^*\\
&=\sum_{P_\pi\sigma W_\lambda: \,\sigma\in\Sigma_{\lambda,\pi}} \mathrm{cor}_{P_\pi\cap W_\lambda^\sigma}^{\Sigma_{k}}\circ  \mathrm{res}^{W_\lambda^\sigma}_{P_\pi\cap W_\lambda^\sigma}\circ p_{\lambda
\sigma^{-1}}^*\circ  \sigma_\lambda^*\\
&=\sum_{P_\pi\sigma W_\lambda: \,\sigma\in\Sigma_{\lambda,\pi}} [W_\lambda^\sigma:P_\pi\cap W_\lambda^\sigma] \cdot\mathrm{cor}_{W_\lambda^\sigma}^{\Sigma_{k}}\circ p_{\lambda\sigma^{-1}}^*\circ  \sigma_\lambda^*\\
&=\sum_{P_\pi\sigma W_\lambda: \,\sigma\in\Sigma_{\lambda,\pi}} [W_\lambda^\sigma:P_\pi\cap W_\lambda^\sigma] \cdot\mathrm{cor}_{W_\lambda}^{\Sigma_{k}}\circ p_{\lambda}^*\\
&=\sum_{P_\pi\sigma W_\lambda: \,\sigma\in\Sigma_{\lambda,\pi}} [W_\lambda:P_\pi^{\sigma^{-1}}\cap W_\lambda] \cdot\mathrm{cor}_{W_\lambda}^{\Sigma_{k}}\circ p_{\lambda}^*.
\end{align*}}Here we have used the double coset formula in the first line and basic properties of (co)restriction in the third, fourth, sixth, and seventh---recalling that $i_\pi$ and $p_\pi$ are $P_\pi$-equivariant---and the ninth is elementary. The second and eighth follow from the commutative diagram
\[\xymatrix{
F_k\ar[d]_-{p_{\lambda\sigma^{-1}}}\ar[r]^-\sigma&F_k\ar[d]^-{p_\lambda}\\
F_{\lambda\sigma^{-1}}\ar[r]^-{\sigma_\lambda}&F_\lambda
}\] For the fifth, we first appeal to Lemma \ref{lem:fitting conditions}(2) to guarantee that $i_\pi^*\circ p_{\lambda\sigma^{-1}}^*\circ\sigma_\lambda^*$ vanishes on $H^{d(\lambda)}(F_{\lambda})$ unless $\lambda\sigma^{-1}$ fits into $\pi$, which is to say that $\sigma\in\Sigma_{\lambda,\pi}$. For the remaining terms, we appeal to Lemma \ref{lem:fitting conditions}(1).
\end{proof}

\section{The Cartan--Leray spectral sequence}

In this section, we prove Theorem \ref{thm:decomposition} and Corollary \ref{cor:am}, modulo a certain double coset calculation that we defer to the Section \ref{section:double cosets}. We work with integer coefficients in Section \ref{section:first observations} before switching in Section \ref{section:AM theorem} to working over a field of fixed prime characteristic $p$.

\subsection{First observations}\label{section:first observations}

To begin, we write $E(k)$ for the spectral sequence obtained by applying Corollary \ref{cor:cartan leray} to the projection $F_k\to B_k$, a regular covering with deck group $\Sigma_k$. Consulting Corollary \ref{cor:partition decomposition}, we obtain the following description of the second page of this spectral sequence.

\begin{corollary}\label{cor:second page}
For every $k\geq0$, there is a canonical isomorphism
\[E(k)_2\cong\bigoplus_{\lambda\vdash k} H^*(\Sigma_k; M_\lambda).\]
\end{corollary}

We next establish a vanishing range on the final page of the spectral sequence. The main input to this range is the following result, which is essentially a theorem of Nakaoka, as the proof makes clear.

\begin{theorem}\label{thm:upper bound}
For $i> (n-1)(k-1)$, we have $H_i(B_k;\mathbb{Z})=0$.
\end{theorem}
\begin{proof}
By Poincar\'{e} duality, it suffices that $H^{nk-i}_c(B_k;\mathbb{Z}^w)=0$ for $i$ in the claimed range. Since $F_k$ is orientable with free $\Sigma_k$-action, we may calculate this cohomology by means of the spectral sequence of Corollary \ref{cor:cs cartan leray}, so it suffices to establish the vanishing of 
\[H^{nk-i}_c(F_k;\mathbb{Z})\cong H^{nk-i}((S^n)^{\wedge k},\Delta;\mathbb{Z})\]
in this range, where $\Delta$ denotes the image in the smash product of the fat diagonal in $(S^n)^k$. Since $S^n$ is $(n-1$)-connected, this group vanishes by \cite[Prop. 4.3]{Nakaoka:CSP} for $nk-i<n+k-1$, as desired.
\end{proof}

\begin{corollary}\label{cor:vanishing}
For any $s+t>(n-1)(k-1)$, we have $E(k)_\infty^{s,t}=0$.
\end{corollary}

\begin{remark}
Building as we do on the foundational work of Nakaoka, Kallel in \cite[Thm. 1.1]{Kallel:SPDHDCS} establishes a result for more general manifolds that specializes to the $n$ even case of Theorem \ref{thm:upper bound}, which is the case in which $B_k$ is orientable, as well as a corresponding result in twisted cohomology in the odd dimensional case. Although we have not checked the details, it seems likely that our spectral sequence technique serves to remove the orientability assumptions in Kallel's result.
\end{remark}

\subsection{The Arone--Mahowald theorem}\label{section:AM theorem} The purpose of this section is to prove Corollary \ref{cor:am} assuming the decomposition theorem. We begin with a few basic observations on special partitions and subordinacy. We refer the reader to the introduction for a reminder of the meaning of these concepts.

\begin{lemma}\label{lem:specialness}
Given a partition $\lambda\vdash k$, the following conditions are equivalent.
\begin{enumerate}
\item $\lambda$ is special.
\item $(\lambda_i)$ is special for every $i$.
\item $\nu$ is special for some $\nu\preccurlyeq \lambda$.
\item $\nu$ is special for some $\lambda\preccurlyeq\nu$.
\item $\nu$ is special for every $\nu\preccurlyeq \lambda$.
\item $\nu$ is special for every $\lambda\preccurlyeq \nu$.
\end{enumerate}
\end{lemma}
\begin{proof}
The first and second conditions are equivalent essentially by definition, and the first clearly implies the third and fourth. Since the fifth and sixth each imply the first, it suffices to show that the third implies the fifth and the fourth implies the sixth. For this, it suffices to observe, assuming that $\lambda$ is obtained from $\nu$ by splitting a block of size $pr$ into $p$ blocks of size $r$, that $\lambda$ is special if and only if $\nu$ is so; indeed, $r$ is a power of $p$ if and only if $pr$ is a power of $p$ (resp. twice a power of $p$), and the remaining blocks of both partitions coincide.
\end{proof}

\begin{lemma}\label{lem:unique maximals}
Given a partition $\lambda\vdash k$, there are unique partitions $\mu$ and $\nu$ such that $\mu$ is maximal, $\nu$ is minimal, and $\nu\preccurlyeq\lambda\preccurlyeq \mu$.
\end{lemma}
\begin{proof}
For the unique maximal, it suffices to show that, if $\lambda$ is subordinate to $\mu_1$ and $\mu_2$, then $\mu_1$ and $\mu_2$ have a common upper bound. By induction, we may assume without loss of generality that, for each $i\in \{1,2\}$, the partition $\lambda$ is obtained from $\mu_i$ by splitting a block of size $pr_i$ into $p$ blocks of size $r_i$. If $r_1=r_2$, then it follows that $\mu_1=\mu_2$; otherwise, $\mu_2$ contains at least $p$ blocks of size $r_1$, and we obtain a common upper bound by replacing these $p$ blocks with a single block of size $pr_1$.

The nearly identical argument for the unique minimal is left to the reader.
\end{proof}

We now come to the proof of the Arone--Mahowald theorem, which, in our formulation, is the assertion that $H^*(\Sigma_k;M_\lambda)$ vanishes in positive degrees if $\lambda$ is not special.

\begin{proof}[Proof of Corollary \ref{cor:am}] From Shapiro's lemma, Proposition \ref{prop:induction description}, and the K\"{u}nneth theorem, we have the isomorphisms 
\[H^*(\Sigma_k;M_\lambda)\cong H^*(W_\lambda; H^{d(\lambda)}(F_\lambda))\cong H^*\left(W_\lambda; \bigotimes_{i=1}^{\ell(\lambda)}H^{d(\lambda_i)} (F_{\lambda_i})\right).\] Applying Corollary \ref{cor:lhs} to $P_\lambda\leq W_\lambda$, and recalling that $P_\lambda$ is the product of the $\Sigma_{\lambda_i}$, the claim follows for $\lambda\neq(k)$ by induction and Lemma \ref{lem:specialness}. In the case $\lambda=(k)$, we claim that $H^*(\Sigma_k;M_k)\subseteq E(k)_2$ supports no nonzero differential; indeed, Theorem \ref{thm:decomposition} grants that the image of $H^*(\Sigma_k;M_k)$ under any differential lies in $\bigoplus_{\nu\prec (k)}H^*(\Sigma_k;M_\nu)$, but $\nu$ is not special for $\nu\prec(k)$ by Lemma \ref{lem:specialness}, so $H^*(\Sigma_k;M_\nu)=0$. Since $H^s(\Sigma_k;M_k)$ lies in bidegree $(s, (k-1)(n-1))$, Corollary \ref{cor:vanishing} now implies the claim.
\end{proof}

Note that, in this argument, the claim for $\lambda\neq (k)$ only uses the validity of Theorem \ref{thm:decomposition} for $E(\ell)$ with $\ell<k$.

\subsection{The decomposition theorem} Recall that the main claim of Theorem \ref{thm:decomposition} is the existence, for every $\mu\vdash k$ maximal under subordinacy, of a spectral subsequence $E(k,\mu)\subseteq E(k)$, whose second page, under the isomorphism of Corollary \ref{cor:second page}, consists of the summands indexed by $\lambda\preccurlyeq\mu$. The key step in the proof, therefore, is the following result.

\begin{lemma}\label{lem:induction step}
If $\mu\vdash k$ is maximal under subordinacy, then the canonical projection of $E(k)_2$ onto $E(k,\mu)_2$ extends to a map $\rho_\mu:E(k)\to E(k)$ of spectral sequences.
\end{lemma}

Assuming this result, the theorem follows easily from our earlier general work regarding spectral subsequences.

\begin{proof}[Proof of Theorem \ref{thm:decomposition}]
By Lemma \ref{lem:unique maximals}, the set of partitions of $k$, partially ordered under subordinacy, is the disjoint union of the lower sets of its maximal elements. Thus, we have 
\[E(k)_2\cong \bigoplus_{\lambda\vdash k} H^*(\Sigma_k;M_\lambda)=\bigoplus_\mu\bigoplus_{\lambda\preccurlyeq\mu} H^*(\Sigma_k;M_\lambda).\] The theorem now follows from Lemma \ref{lem:induction step} by repeated use of Proposition \ref{prop:projection extends}.
\end{proof}

For the proof of Lemma \ref{lem:induction step}, we will use the maps in and maps out introduced above. Given a set composition $\pi$ of $k$, we write $E(\pi)$ for the spectral sequence obtained by applying Corollary \ref{cor:cartan leray} to the projection $F_\pi\to F_\pi/P_\pi$. Using Proposition \ref{prop:naturality}, the maps $i_\pi$ and $p_\pi$ studied in Section \ref{section:maps in maps out} induce maps of spectral sequences $E(\pi)\to E(k)$ and $E(k)\to E(\pi)$ given on the second pages by the natural maps $i_\pi^*\circ \mathrm{res}_{P_\pi}^{\Sigma_k}$ and $\mathrm{cor}_{P_\pi}^{\Sigma_k} \circ p_\pi^*$, respectively.

Our construction will be premised on composing these maps, the result of which we characterized in Proposition \ref{prop:endomorphism constant} in terms of the constants $c_{\lambda,\pi}$, concerning which we have the following key result, whose proof we defer to Section \ref{section:double cosets}.

\begin{lemma}\label{lem:constant nonzero}
If $\lambda\preccurlyeq\mu$, $\lambda$ is special, and $\mu$ is maximal under subordinacy, then $c_{\lambda,\mu}$ is invertible mod $p$.
\end{lemma}

Assuming this result for the moment, we may complete the argument.

\begin{proof}[Proof of Lemma \ref{lem:induction step}]
We prove the lemma by strong induction on $k$, the case $k=1$ being trivial. Fixing a maximal $\mu$, we may assume by Proposition \ref{prop:projection extends} that $\mu\neq (k)$. Consider the map $\rho_\mu$ defined as the composite
\[E(k)\xrightarrow{} E(\mu)\cong \bigotimes_{i=1}^{\ell(\mu)} E(\mu_i)\xrightarrow{\bigotimes \rho_{(\mu_i)}} \bigotimes_{i=1}^{\ell(\mu)} E(\mu_i)\cong E(\mu)\xrightarrow{} E(k),\] where the first map is induced by $i_\mu$, the second is induced by $p_\mu$, and $\rho_{(\mu_i)}$ is the map of spectral sequences obtained by induction. Note that we use our assumption that $\mu\neq (k)$ to guarantee that $\mu_i< k$ for each $i$. By Proposition \ref{prop:projection extends}, it suffices to show that this composite restricts to $H^*(\Sigma_k;M_\lambda)$ trivially or as an automorphism according to whether $\lambda\not\preccurlyeq \mu$ or $\lambda\preccurlyeq\mu$. In the latter case, this claim follows from Proposition \ref{prop:endomorphism constant} and Lemma \ref{lem:constant nonzero}. 

In the former case, reasoning as in the proof of Proposition \ref{prop:endomorphism constant}, we have the equality
{\begin{align*}
i_\mu^*\circ\mathrm{res}_{P_\mu}^{\Sigma_k}\circ\mathrm{cor}^{\Sigma_k}_{W_\lambda}\circ p_\lambda^*
&=\sum_{P_\mu\sigma W_\lambda: \,\sigma\in\Sigma_{\lambda,\mu}} i_\mu^*\circ   p_{\lambda\sigma^{-1}}^*\circ \mathrm{res}^{W_\lambda^\sigma}_{P_\mu\cap W_\lambda^\sigma}\circ\sigma_\lambda^*.
\end{align*}} Writing $f_\sigma$ for the summand of this expression indexed by $\sigma\in \Sigma_{\lambda,\mu}$, and recalling that $\mathrm{cor}_{W_\lambda}^{\Sigma_k}\circ p_\lambda^*$ is an isomorphism of $H^*(W_\lambda; H^{d(\lambda)}(F_\lambda))$ onto $H^*(\Sigma_k;M_\lambda)$ by Proposition \ref{prop:induction description} and Shapiro's lemma, it suffices to show that $f_\sigma$ maps $H^*(W_\lambda;H^{d(\lambda)}(F_\lambda))$ into the kernel of $\bigotimes_i \rho_{(\mu_i)}$. 

Fixing $\sigma$, we have that $\lambda\sigma^{-1}$ fits into $\mu$ by definition, thereby inducing a partition of each block of $\mu$, say $\eta_i$ of $\mu_i$, and the map $p_{\lambda\sigma^{-1}}\circ i_\mu$ factors up to homotopy as the composite
\[F_{\mu}\cong \prod_{i=1}^{\ell(\mu)} F_{\mu_i}\xrightarrow{\prod_i p_{\eta_i}} \prod_{i=1}^{\ell(\mu)} F_{\eta_i}\cong F_{\lambda\sigma^{-1}}.\] It follows that the image of $H^*(W_\lambda;H^{d(\lambda)}(F_\lambda))$ under $f_\sigma$ lies in the summand of $\bigotimes_{i=1}^{\ell(\mu)} E(\mu_i)_2$ indexed by $(\eta_1,\ldots, \eta_{\ell(\mu)})$.  If $\eta_i\not\preccurlyeq (\mu_i)$ for some $i$, then this summand lies in the kernel of $\bigotimes_i \rho_{(\mu_i)}$, as desired. But $\eta_i\preccurlyeq (\mu_i)$ for every $i$ if and only if $\lambda\preccurlyeq \mu$, so the claim follows.
\end{proof}

\section{A double coset calculation}\label{section:double cosets}

The goal of this section is to prove Lemma \ref{lem:constant nonzero}, which concerns the constant $c_{\lambda,\mu}$ from Section \ref{section:maps in maps out}. Given that this constant is defined as a sum over double cosets represented by elements of the set $\Sigma_{\lambda,\mu}\subseteq\Sigma_k$ of fittings, our first goal will be to find a combinatorial interpretation for this set of double cosets.

\subsection{Block distributions} We begin with the following general combinatorial concept.

\begin{definition}
Let $\rho$ and $\pi$ be set compositions. A \emph{(distinguishable) block distribution} of $\rho$ into $\pi$ is a function $d:\langle \ell(\rho)\rangle \to \langle \ell(\pi)\rangle$. The \emph{weight} of $d$ is the tuple \[w(d)=\left(\sum_{i\in d^{-1}(j)} |\rho_i|\right)_{j=1}^{\ell(\pi)}.\]
\end{definition}

We write $\widetilde B_{\rho,\pi}$ for the set of all distinguishable block distributions of $\rho$ into $\pi$. This set inherits a right action of $W_\rho$ via the action on $\langle \ell(\rho)\rangle$.

\begin{definition}
Let $\rho$ and $\pi$ be set compositions. The set of \emph{indistinguishable block distributions of $\rho$ into $\pi$} is the set of orbits $B_{\rho,\pi}=\widetilde B_{\rho,\pi}/W_\rho$.
\end{definition}

It is not hard to see that the action of $W_\rho$ preserves weight, so that we may speak of the weight of an indistinguishable block distribution. We write $\widetilde B_{\rho,\pi}(w)$ and $B_{\rho,\pi}(w)$ for the respective subsets of block distributions of fixed weight $w\in\mathbb{Z}_{\geq0}^{\ell(\pi)}$. We will be particularly interested in the case $w=(|\pi_i|)_{i=1}^{\langle\ell(\pi)\rangle}$, which we abusively indicate by $B_{\rho,\pi}(\pi)$, and so on. In order to better understand these sets, we connect them to the fittings of Definition \ref{def:fitting}.

\begin{proposition}\label{prop:cosets and blocks}
Let $\rho$ and $\pi$ be set compositions of $\langle k\rangle$. The assignment $\sigma\mapsto d_\sigma$ induces the indicated bijections in the following commutative diagram:
\[\xymatrix{
P_\pi\backslash \Sigma_{\rho,\pi}\ar[d]\ar[r]^-\cong& \widetilde B_{\rho,\pi}(\pi)\ar[d]\\
P_\pi\backslash \Sigma_{\rho,\pi}/W_\rho\ar[r]^-\cong&  B_{\rho,\pi}(\pi).
}\]
\end{proposition}
\begin{proof}
Since the assignment $\sigma\mapsto d_\sigma$ is clearly $W_\rho$-equivariant, the claim regarding the top arrow implies the claim regarding the bottom after passing to orbits. For the first, descent to the quotient follows readily from the commutative diagram
\[\xymatrix{
\langle k\rangle\ar[d]_-\rho \ar[r]^-\sigma& \langle k\rangle\ar[d]^-\pi\ar[r]^-{\tau}&\langle k\rangle\ar[dl]^-\pi\\
\langle \ell(\rho)\rangle\ar[r]^-{d_\sigma}& \langle \ell(\pi)\rangle,
}\] where $\tau\in P_\pi$. For bijectivity, it suffices to show that any two choices of $\sigma$ in this diagram differ by the action of $P_\pi$. But, by choosing $\tau$ appropriately, we may arrange that $\sigma$ respects the natural ordering within each block of $\rho$ and the natural ordering of the blocks of $\rho$ within each block of $\pi$. Since these two requirements uniquely specify $\sigma$, the claim follows.
\end{proof}

\begin{corollary}\label{cor:blocks and index}
For set compositions $\rho$ and $\pi$ of $\langle k\rangle$, we have
\[|\widetilde B_{\rho,\pi}(\pi)|=\sum_{P_\pi\sigma W_\rho,\,\sigma\in \Sigma_{\rho,\pi}}[W_\rho:P^{\sigma^{-1}}_\pi\cap W_\rho].\]
\end{corollary}
\begin{proof}
By Proposition \ref{prop:cosets and blocks} and the orbit-stabilizer theorem, it suffices to note that $P^{\sigma^{-1}}_{\pi} \cap W_{\rho}$ is the stabilizer of $P_\pi\sigma$ under the action of $W_\rho$ on $P_\pi\backslash \Sigma_{\rho,\pi}$.
\end{proof}

We close this section by recording a generating function for the cardinalities of the sets of distinguishable block distributions of fixed weight.

\begin{proposition}\label{prop:polynomial}
Given $w\in\mathbb{Z}_{\geq0}^{\ell(\rho)}$, the cardinality of $\widetilde B_{\rho,\pi}(w)$ is equal to the coefficient of $x_1^{w_1}\cdots x_{\ell(\pi)}^{w_{\ell(\pi)}}$ in the polynomial 
\[q_{\rho,\pi}=\prod_{i=1}^{\ell(\rho)}\sum_{j=1}^{\ell(\pi)} x_j^{|\rho_i|}.\]
\end{proposition}

\subsection{Proof of Lemma \ref{lem:constant nonzero}}
Assume first that $n$ is odd. In this case, the blocks of $\lambda$ and of $\mu$ are powers of $p$ by specialness (see Lemma \ref{lem:specialness}). Moreover, writing the base $p$ expansion $k=\sum_{r=0}^s a_r p^r$ with $0\leq a_r<p$, maximality and specialness force $\mu$ to be the partition with $a_r$ many blocks of size $p^r$ for each $r$.

Now, by Corollary \ref{cor:blocks and index} and Proposition \ref{prop:polynomial}, the quantity in question is equal to the coefficient of $x_1^{\mu_1}\cdots x_{\ell(\mu)}^{\mu_{\ell(\mu)}}$ in the polynomial 
\begin{align*}
q_{\lambda,\mu}&=\prod_{i=1}^{\ell(\lambda)}\sum_{j=1}^{\ell(\mu)} x_j^{\lambda_i}\\
&\equiv\prod_{i=1}^{\ell(\lambda)}\left(\sum_{j=1}^{\ell(\mu)} x_j\right)^{\lambda_i}\mod p\\
&=\left(\sum_{j=1}^{\ell(\mu)} x_j\right)^{k},
\end{align*} where we have used the freshman's dream in the second line. Thus, from the multinomial theorem, we conclude in this case that 
\begin{align*}
\sum_{P_\mu\sigma W_\lambda,\,\sigma\in \Sigma_{\lambda,\mu}}[W_\lambda:P^{\sigma^{-1}}_\mu\cap W_\lambda]&\equiv \binom{k}{\mu_1,\ldots, \mu_{\ell(\mu)}}\mod p\\
&\equiv \prod_{r=0}^{s}a_r!\mod p,
\end{align*}
where the last line follows from Lucas' theorem for multinomial coefficients. Since $p$ does not divide $a_r!$ for $0\leq a_r< p$, the claim follows in this case.

Assuming instead that $n$ is even and that every block of $\lambda$ (and hence of $\mu$) is twice a power of $p$, the same argument applies after substituting $x_j^2$ for $x_j$ and $k/2$ for $k$. 

In the general case, we may write $k=k'+k''$ and regard $\lambda$ as the union of a partition $\lambda'$ of $k$, each of whose blocks is a power of $p$, with a partition $\lambda''$ of $k''$, each of whose blocks is twice a power of $p$. The same remarks apply to $\mu$, and we have $\lambda'\preccurlyeq\mu'$ with $\mu'$ maximal under subordinacy (resp. $\lambda''$, $\mu''$). Arguing as before, we conclude that
\[q_{\lambda,\mu}\equiv \left(\sum_{j=1}^{\ell(\mu)} x_j\right)^{k'}\left(\sum_{j=1}^{\ell(\mu)} x_j^2\right)^{\frac{k''}{2}}\mod p.\] Applying the multinomial theorem twice, a general term in the expansion of this expression takes the form 
\[\binom{k'}{a_1,\ldots, a_{\ell(\mu)}}\binom{\frac{k''}{2}}{b_1,\ldots, b_{\ell(\mu)}}\prod_{j=1}^{\ell(\mu)}x_{j}^{a_j+2b_j}.\]  We may restrict our attention to terms of this form such that $a_j+2b_j=\mu_j$, which implies that $a_j$ is odd, and in particular nonzero, whenever $\mu_j$ is odd.  By Lucas' theorem, since we are working mod $p$, we may further restrict our attention to the terms in which each $p$-adic digit of $k'$ is the sum of the corresponding $p$-adic digits of the $a_j$. But the odd $\mu_j$ form the $p$-adic expansion of $k'$, as above, so these two observations force $a_j=\mu_j$ for $\mu_j$ odd and $a_j=0$ for $\mu_j$ even, which in turn forces $b_j=\frac{\mu_j}{2}$ for $\mu_j$ even and $b_j=0$ for $\mu_j$ odd.
In this last remaining case, the coefficient of interest is the product of two coefficients already shown to be nonzero mod $p$. 

\bibliographystyle{amsalpha}
\bibliography{references}
\end{document}